\documentclass[12pt,french]{article}
\usepackage{graphicx}
\usepackage[french]{babel} 
\usepackage{a4} 
\usepackage[T1]{fontenc} 
\usepackage[cyr]{aeguill} 
\usepackage{epsfig} 
\usepackage{amsmath, amsthm} 
\usepackage{amsfonts,amssymb}
\usepackage{float} 
\usepackage{url} 
\usepackage{colortbl}
\usepackage{graphics}
\usepackage[latin1]{inputenc}
\usepackage{tikz,tkz-tab}
\usepackage{fancyhdr}
\newtheorem{remark}{Remarque}[section]
\newtheorem{definition}{Définition}[section]
\newtheorem{proposition}{Proposition}[section]
\newtheorem{lemma}{Lemme}[section]

\newtheorem{theorem}{Théorème}[section]
\date{}

\begin{document}
\title{Quelques propriétés qualitatives pour le problème de surfaces de quadrature dans le plan}
\author{Mohammed Barkatou}
\date{}
\maketitle
\begin{abstract}
Dans ce papier, nous commençons par donner une condition nécessaire et suffisante d'existence de solutions pour le problème de surfaces de quadrature
dans le cas où le terme source est une densité supportée par un segment. Ensuite, en utilisant la symétrisation de Steiner continue,
nous montrons que la solution obtenue est symétrique par rapport à l'axe des abscisses et que son bord est analytique.
\end{abstract}

\maketitle {\small {\noindent {\bf Keywords:} Dérivation par rapport au domaine, Principe du maximum, Propriété géométrique de la normale, Symétrisation de Steiner.


\section{Introduction}
On se donne une constante strictement positive $k$ et une fonction $f\in L^2(\mathbb{R}^N)$, positive et à support compact $K$.
Notons par $C$ l'enveloppe convexe de $K$.
Considérons le problème de surfaces de quadrature $QS(f,k)$ suivant :  Trouver un ouvert borné $\Omega$ qui contient strictement $C$
et tel que le problème surdéterminé suivant ait une solution.
\begin{equation*}
 \left\{
\begin{array}{c}
-\Delta u=f \quad \text{dans  }\Omega\\
u=0\text{ sur  }\partial \Omega\\
 \left|
\nabla u\right|=k\; \text{  sur }\partial
 \Omega.
\end{array}
\right.
\end{equation*}
Le problème $QS(f,k)$ trouve son application dans plusieurs domaines de la physique
(free streamlines, jets, Hele-show, electromagnetic shaping, gravitational problems...etc).
Il a été étudié, de différentes façons, par plusieurs auteurs (pour plus de détails concernant les méthodes utiliées voir l'introduction de l'article de Gustafsson et Shahgholian \cite{gs}).

Dans \cite{sh} (voir aussi la Remarque 4.1 ci-dessous), en utilisant essentiellement la méthode des hyperplans mobiles \cite{gwn},
H. Shahgholian montre que si le problème à frontière libre $QS(f,k$) admet une solution $\Omega$ (de classe $C^{2}$), alors

\begin{enumerate}
  \item[(S1)] $\Omega$ contient strictement le convexe compact $C$.
  \item[(S2)] $\partial\Omega\setminus C$ est lipschitzien.
  \item[(S3)] En tout point de $\partial\Omega\setminus C$ où la normale intérieure existe, celle-ci rencontre $C$.
\end{enumerate}
La classe des ouverts qui vérifient ces trois propriétés n'est pas fermée pour la topologie de Hausdorff car l'ensemble des ouverts qui contiennent le compact $C$ ne l'est pas). Par ailleurs, les éléments de cette classe ne sont, en général, pas stables au sens de Keldysh-Wiener \cite{ke}, \cite{wi}.
En effet, si on considère, en deux dimensions, le segment $C=[0,1]\times\{0\}$, alors l'ouvert $\{(x,y\in\mathbb{R}^2,\;x^{2}+y^{2}<1\}$
vérifie les trois propriétés ci-dessus mais il n'est pas stable au sens de Keldysh-Wiener.\\
Dans \cite{ba1}, l'auteur introduit une classe d'ouverts vérifiant :
\begin{enumerate}
  \item[(N1)] $\Omega$ contient l'intérieur de $C$.
  \item[(N2)] $\partial\Omega\setminus C$ est lipschitzien.
  \item[(N3)] En tout point de $\partial\Omega\setminus C$ où la normale intérieure existe, celle-ci rencontre $C$.
  \item[(N4)] Pour toute normale (sélectionnée), $\Delta$, à $C$, $\delta\cap\Omega$ est connexe.
\end{enumerate}
Ces quatre conditions formant ce que l'auteur appelle la Propriété Géométrique de la Normale par rapport à $C$ (plus simplement $C$-GNP). Notons $\mathcal{O}_{C}$
la classe des ouverts bornés de $\mathbb{R}^{N}$ possédant la $C$-GNP.\quad

En utilisant la dérivation par rapport au domaine \cite{sz}, le problème $QS(f,k)$ apparaît comme l'équation d'Euler du problème de minimisation,
sur la classe des ouverts admissibles $\mathcal{O}_{C}$, de la fonctionnelle :

$$J(\Omega)=\int_{\Omega}(|\nabla u|^{2}-2fu+k^{2})dx$$
où $u$ est la solution du problème de Dirichlet sur $\Omega$.\\
Une fois que l'existence d'un minimum $\Omega$ est établie, Pour que celui-ci, soit une solution du problème $QS(f,k)$, il doit contenir strictement $C$
et la contidtion surdéterminée $|\nabla u|=k$ doit être obtenue au moins presque partout sur $\partial\Omega$. Cela ne peut se produire que si nous mettons
des conditions sur les données du problème $f$, $C$ et $k$.\\
Le but de cet article est de donner une condition nécessaire et suffisante dans le cas où le terme source $f$ est une densité supportée par un segement.
En clair, nous allons nous placer dans $\mathbb{R}^2$
et nous prendrons $f=a\delta_{C}$ où $C=[-1,1]\times\{0\}$ ($a$ est un réel strictement positif).\quad

Nous commençerons par donner une condition nécessaire d'existence de solution pour le problème $QS(a\delta_{C},k)$.
Ensuite, en appliquant le principe du maximum aux deux problèmes d'optimisation de formes considérés, nous allons démontrer que la condition nécessaire obtenue
est en fait suffisante pour que le minimum $\Omega$ obtenu contienne strictement le convexe $C$ (cf Théorème 3.1).
Enfin, la dérivation par rapport au domaine nous permettra de récupérer la condition surdéterminée sur $\partial\Omega$ et
la symétrisation de Steiner continue nous donnera le caractère analytique du bord de $\Omega$ ainsi que la symétrie de celui-ci par rapport à l'axe $Ox$
(cf Théorème 3.2).

\section{Résultats préliminaires}
Dans cette section, nous allons énnoncer puis démontrer quelques résultats qui nous permettront de prouver les Théorèmes 3.1 et 3.2.\\
La Proposition 2.3 \cite{ba1} nous donne une caractérisation de la $C$-GNP où on n'a pas besoin d'utiliser la normale. Cette propriété est notée $C$-SP et elle traduit le fait  $$\forall x\in\partial\Omega\setminus C,\quad K_{x}\cap\Omega\neq\emptyset$$  où
$$K_{x}=\{y\in \mathbb{R}^N\;:\;(y-x).(c-x)\leq 0,\;\forall c\in C\}.$$

\begin{proposition}
Soit $\Omega$ un ouvert qui possède la $C$-GNP. Alors hormis les parties de $\partial\Omega\cap C$ (où seulement les champs de vecteurs $V$ tels que $V.\nu\geq 0$ sont autorisés), la $C$-GNP est conservée pour $\partial\Omega\setminus C$, quelque soit le champ de vecteurs $V$ de $\mathbb{R}^N$ dans $\mathbb{R}^N$,
borné et lipschitien, et tout $t\in ]0,1[$.
\end{proposition}

\begin{proof}
Grâce à l'équivalence entre la $C$-GNP et la $C$-SP, nous allons démontrer que si $\Omega$ possède la $C$-GNP alors $\Omega_{t}$ la possède aussi. La condition (N1) est obtenue pour les champs de vecteurs $V$ tels que $V.\nu\geq 0$ alors que les conditions (N2) et (N4) sont conservées par perturbation du domaine $\Omega$ pour tout champ de vecteurs $V$. Il nous reste donc à démontrer que la condition (N3) est aussi vérifiée.\quad

En effet, supposons, par l'absurde, qu'il existe $x_{t}\in\partial\Omega_{t}\setminus C$ tel que $K_{x_{t}}\cap\Omega_{t}\neq\emptyset$. Soit $y_{t}\in K_{x_{t}}\cap\Omega_{t}$, il existe alors $y\in\Omega$, $y=y_{t}-tV(y)$ tel que :
$$\forall c\in C,\;\;(y_{t}-x_{t}).(c-x_{t})\leq 0.$$
Montrons que $y\in K_{x}$. En effet,
\begin{eqnarray*}
    (y-x).(c-x)&=&(y_{t}-tV(y)-x_{t}+tV(x)).(c-x_{t}+tV(x))\\
    &=&(y_{t}-x_{t}+t(V(y)-V(x))).(c-x_{t}+tV(x))\\
    &=&(y_{t}-x_{t}).(c-x_{t})+\epsilon(t)
\end{eqnarray*}

où $\epsilon(t)=t(y_{t}-x_{t}).V(x)+t(V(y)-V(x)).(c-x_{t})+t^2(V(y)-V(x)).V(x)$ qui, comme $t$, tend vers $0$. Obtenant ainsi la contradiction.

\end{proof}

\begin{lemma}
Soit $\Omega$ un ouvert de $\mathbb{R}^2$ qui contient le segment $C=[-1,1]\times\{0\}$. Si $\Omega$ a la $C$-GNP alors $\Omega$ est convexe dans la direction $Oy$.
\end{lemma}
La preuve de ce lemme est une conséquence immédiate de la définition de la $C$-GNP, puisque les normales sélectionnées sont les droites orthogonales au segment $C$.
\begin{remark}
Comme $C$ est contenu dans $\Omega$, les conditions 2. et 3. de la $C$-GNP suffisent pout montrer ce lemme. En effet, soit $H^{+}$ et $H^{-}$ les demi-espaces supérieur et inférieur séparés par l'axe $Ox$. Soit $x$ un point de $\partial\Omega$. Comme celui-ci vérifie la $C$-GNP, il vérifie la $C$-SP et par conséquent le segment vertical au-dessus de $x$ est inclus dans le cone $K_{x}$. De même, s'il existait un point $z$ sur le segment vertical au-dessous de $x$, on aurait $x\in K_{z}$ contredisant ainsi le Lemme 2.4 \cite{ba1}. Donc $\Omega$ est convexe suivant la direction $Oy$ et $\partial\Omega\cap H^{+}$ et $\partial\Omega\cap H^{-}$ sont des graphes.
\end{remark}
\begin{definition}
Un arc de cercle $\gamma$ est dit de Type I si
\begin{enumerate}
  \item $\gamma$ est centré en $(-1,0)$ et est contenu dans $\{(x,y)\in\mathbb{R}^2,\; x\leq -1\}$, ou
  \item $\gamma$ est centré en $(1,0)$ et est contenu dans $\{(x,y)\in\mathbb{R}^2,\; x\geq 1\}$
\end{enumerate}
$\gamma$ est dit de Type II, sinon.
\end{definition}

\begin{definition}(Symétrisation de Steiner)
Soit $\Omega$ un ouvert de $\mathbb{R}^2$, convexe dans la direction $Oy$. Soient $\alpha$ un réel fixé et $\Omega_{\alpha}$ le segment $\{(\alpha,y),\;(\alpha,y)\in\Omega$. La symétrisation de Steiner consiste à recentrer chaque segment $\Omega_{\alpha}$ sur l'axe $Ox$ et le symétrisé $\Omega^{*}$ de $\Omega$ est défini par
$$\Omega^{*}=\{(\alpha,t),\;\;\alpha\in\mathbb{R},\;|t|<\frac{|\Omega_{\alpha}|}{2}\}.$$
\end{definition}
La symétrisation de Steiner ne conserve pas la $C$-GNP (voir la Remarque 4.2 ci-dessous). Pour y remédier, nous allons utiliser la symétrisation de Steiner continue (Proposition 2.6 \cite{ba1}).
\begin{definition}(Symétrisation de Steiner continue)
Soit $\Omega$ un ouvert de $\mathbb{R}^2$, convexe dans la direction $Oy$. La symétrisation de Steiner continue consiste à recentrer chaque segment
$[y_{1},y_{2}]$ parallèle à l'axe $Oy$ ($y_{1}$ et $y_{2}$ appartenant à $\partial\Omega$) avec une vitesse égale à la distance du centre de $[y_{1},y_{2}]$
à la droite $x=0$. C'est à dire que si $\partial\Omega$ est donné par deux fonctions $\phi_{1}$ et $\phi_{2}$ alors pour $t\in [0,1]$ son symétrisé $\Omega^{t}$ sera donné par les fonctions $\phi^{t}_{1}$ et $\phi^{t}_{2}$ définies par :
$$\phi^{t}_{1}=\phi_{1}-\frac{t}{2}(\phi_{1}-\phi_{2}),\;\; \phi^{t}_{2}=\phi_{2}+\frac{t}{2}(\phi_{1}-\phi_{2})$$
\end{definition}
\begin{proposition}
Soit $\Omega$ un ouvert de $\mathbb{R}^2$ qui vérifie la $C$-PGN, contient le segment $C$ et tel que $\partial\Omega$ ne contient pas d'arc de Type II.
Alors, pour $t$ assez petit, $\Omega^{t}$ admet la $C$-PGN.
\end{proposition}
\begin{proof}
Pour $t$ petit, le symétrisé $\Omega^{t}$, par la symétrisation de Steiner continue, d'un ouvert $\Omega$ qui possède la $C$-GNP,
vérifie la même propriété pour les points de son bord dont la normale (lorsqu'elle existe) rencontre l'intérieur relatif de $C$.
Pour les points de $\partial\Omega^{t}$ qui sont sur des arcs de cercle, nous allons démontrer que pour les arcs de Type I,
la $C$-GNP est conservée. En effet, Supposons que $\partial\Omega$ est donné par les fonctions $\phi_{1}$ et $\phi_{2}$,
alors, pour $t$ petit, $\partial\Omega^{t}$ est donné par les fonctions $\phi^{t}_{1}$ et $\phi^{t}_{2}$. Pour $\Omega^{t}$, la $C$-GNP se traduit par :
$$-1\leq x+\phi^{t}_{i}(x)(\phi_{i}^{t})'(x)\leq 1,\;\; i=1,\;2.$$
Il s'agit donc de vérifier que
$$-1\leq x+\phi_{1}(x)\phi_{1}^{'}(x)-\eta(t)\leq 1,$$
$$\eta(t,x)=-\frac{t}{2}[2\phi_{1}\phi^{'}_{1}(x)-(\phi_{1}^{'}(x)\phi_{2}(x)+\phi_{1}(x)\phi^{'}_{2}(x))]+
\frac{t^{2}}{4}(\phi_{1}(x)-\phi_{2}(x))(\phi^{'}_{1}(x)-\phi^{'}_{2}(x)).$$
Plaçons-nous maintenant en un point de $\partial\Omega$, d'abscisse $x_{0}$ tel que la normale intérieure coupe le segment $C$, par exemple, en $(-1,0)$. On a
$$x_{0}+\phi_{1}(x_{0})\phi^{'}_{1}(x_{0})=-1.$$
Par ailleurs, la $C$-GNP, pour $\phi_{2}$ fournit
$$\phi_{2}^{'}(x_{0})\geq -\frac{1+x_{0}}{\phi_{2}(x_{0})}.$$
Au point d'abscisse $x_{0}$, nous obtenons donc
$$\phi_{1}^{'}(x_{0})\phi_{2}(x_{0})+\phi_{1}(x_{0})\phi^{'}_{2}(x_{0})\geq -(1+x_{0})(\frac{\phi_{2}(x_{0})}{\phi_{1}(x_{0})}+
\frac{\phi_{1}(x_{0})}{\phi_{2}(x_{0})}).$$
Or $$\frac{\phi_{2}(x_{0})}{\phi_{1}(x_{0})}+\frac{\phi_{1}(x_{0})}{\phi_{2}(x_{0})}\geq 2$$
(et même $>2$ si $\phi_{2}(x_{0})\neq\phi_{1}(x_{0})$, l'égalité correspond au cas où on ne bouge pas). Donc, si $1+x_{0}>0$ on a bien
$$\phi_{1}^{'}(x_{0})\phi_{2}(x_{0})+\phi_{1}(x_{0})\phi^{'}_{2}(x_{0})\geq -2(1+x_{0})=2\phi_{1}(x_{0})\phi_{2}^{'}(x_{0}).$$
Donc le terme $\eta(t,x_{0})$ sera négatif. Par suite l'inégalité ci-dessus est vérifiée au point d'abscisse $x_{0}$.
Par ailleurs, comme $x_{0}+\phi_{1}(x_{0})\phi_{1}^{'}(x_{0})=-1,$ et $t$ est assez petit, alors
$$-1\leq x_{0}+\phi_{1}(x_{0})\phi_{1}^{'}(x_{0})-\eta(t,x_{0})\leq 1.$$
\end{proof}

\section{Résultats principaux}
Sans conditions sur $f$ et $k$, le problème $QS(f,k)$ peut ne pas avoir de solutions. En effet, si $(\Omega,u)$ en est une solution telle que $\Omega$ est à
bord lipschitzien et $u\in H^{2}(\Omega)$, alors la formule de Green nous donne la condition nécessaire suivante :
$$\int_{C}f=k|\partial\Omega|.$$
($|\partial\Omega|$ étant le périmètre de $\Omega$). En particulier, si le rapport $\frac{\int_{C}f}{k}$ est trop petit pour qu'on puisse trouver un ouvert $\Omega$ englobant le convexe $C$ dont le périmètre vérifie $|\partial\Omega|=\frac{\int_{C}f}{k}$, c'est à dire si $|\partial C|>\frac{\int_{C}f}{k}$, il est claire que le problème $QS(f,k)$ ne peut pas avoir de solution. Cela signifie que le minimum obtenu pour la fonctionnelle $J$ n'est pas une solution du problème $QS(f,k)$ puisque le convexe $C$ n'y est pas inclus strictement. Autrement dit, dans le cas qui nous intéresse ici, le minimum $\Omega$ peut venir toucher le segment $C$ (en un point anguleux ou de rebroussement voire un segment).
\begin{theorem}
$\Omega$ est solution du problème $QS(a\delta_{C},k)$ si et seulement si $a>2k$.
\end{theorem}
Dans son article \cite{fr}, A. Friedman a montré qu'en deux dimensions, si l'ouvert $\Omega$ solution du problème $QS(a\delta_{C},k)$ existe et est à bord lipschitzien alors il est analytique. Le bord du minimum que nous obtenons ici est analytique par morceaux :
il contient $\Gamma_{1}=\partial\Omega\setminus C$ (où on sait que $|\nabla u|=k$) qui est analytique et $\Gamma_{2}$ constitué des arcs de cercles de Type I et II (une réunion de parties analytiques). Il restera alors les points de jonction entre les parties de $\Gamma_{1}$ et celles de $\Gamma_{2}$. Par conséquent, si on arrive à avoir des conditions sur $a$ et $k$ pour que $\partial\Omega$ soit analytique en ces points, alors $\Omega$ sera soit un cercle, soit $\partial\Omega=\Gamma_{1}$ qui fournirait une solution du problème $QS(a\delta_{C},k)$.\\

Le théorème qui suit dit, entre autre, que si le bord de $\Omega$ ne contient pas d'arcs de Type II, alors il est analytique dès que $a>2k$.

\begin{theorem}
Soit $\Omega$ un minimum de la fonctionnelle $J$ sur $\mathcal{O}_{C}$. Supposons que $\partial\Omega$ ne contient aucun arc de cercle centré aux extrémités de $C$ et que $u\in H^2(V_{\partial\Omega})$ ($V_{\partial\Omega}$ un voisinqge de $\partial\Omega$). Alors $\Omega$ est une solution classique (i.e $|\nabla u|=k$ sur $\partial\Omega$) du problème $QS(a\delta_{C},k)$ qui est symétrique par rapport à l'axe $Ox$ et dont le bord $\partial\Omega$ est analytique.
\end{theorem}

\subsection{Preuve du Théoreme 3.1}
La démonstration de ce théorème se fait en prouvant les deux propositions suivantes.
\begin{proposition}
Soit $(\Omega,u)$ une solution du problème $QS(a\delta_{C},k)$. Si $\Omega$ est à bord lipschitzien et $u\in H^{2}(\Omega\setminus C)$ alors $a>2k$.
\end{proposition}

\begin{proof}
soient $\epsilon\in[0,1]$ et $V_{\epsilon}=[-1-\epsilon,1+\epsilon]\times[-\epsilon,\epsilon]$. Posons $\Omega_{\epsilon}=\Omega\setminus V_{\epsilon}$.
$u$ est harmonique sur $\Omega_{\epsilon}$, donc
$$0=\int_{\Omega_{\epsilon}}\Delta u=\int_{\partial\Omega}\frac{\partial u}{\partial\nu}+\int_{\partial V_{\epsilon}}-\frac{\partial u}{\partial\nu}.$$
En écrivant $u=h-\frac{a}{2}|y|$ ($h$ étant une fonction harmonique sur $\Omega$) et tendant $\epsilon$ vers $0$, on obtient
$$\lim_{\epsilon\rightarrow 0}\int_{\partial V_{\epsilon}}\frac{\partial u}{\partial\nu}=-2a.$$
Mais $-\frac{\partial u}{\partial\nu}=k$ sur $\partial\Omega$, donc $k|\partial\Omega|=2a$.
$C$ étant strictement inclus dans $\Omega$, $|\partial\Omega|>4$, d'où le résultat.
\end{proof}

\begin{proposition}
Si $a>2k$ alors il existe un ouvert borné $\Omega$ qui contient strictement la boule $B(O,1)$ et tel que $\Omega$ minimise
la fonctionnelle $J$ sur la classe des ouverts vérifiant la $C$-PGN.
\end{proposition}
Soit $B$ la boule (ouverte) unité de $\mathbb{R}^2$. Posons
$$\mathcal{O}_{C,B}=\{B\subset\omega\subset D,\;\;\omega\in\mathcal{O}_{C}\}.$$
Par le principe du maximum, la fonctionnelle
$$J(\omega)=\int_{\omega}(|\nabla u|^{2}-2fu+k^{2})dx$$
est minorée inférieurement et par le Theorème 4.3 \cite{ba1}, elle admet au moins un minimum $\Omega$.
Par ailleurs, si le minimum $\Omega$ est de classe $C^2$, la Proposition 2.1 ci-dessus appliquée à $\Omega$ et $B$ nous permet d'écrire
\begin{enumerate}
  \item[(C1)] $|\nabla u|\leq k$ sur $\partial\Omega\cap\partial B$, et
  \item[(C2)] $|\nabla u|=k$ sur $\partial\Omega\setminus\partial B$.
\end{enumerate}
Posons,
$$\mathcal{O}_{\Omega}=\{\omega\subset \Omega,\;\;\omega\in\mathcal{O}_{C,B}\},\;\text{et}$$
$$F(\omega)=\int_{\omega}(|\nabla u|^{2}-2fu+(\frac{a}{2})^{2})dx.$$
La fonctionnelle $F$ admet au moins un minimum $\Omega^{*}$ sur $\mathcal{O}_{\Omega}$. Si $\Omega^{*}$ est de classe $C^2$,
\quad
\begin{itemize}
  \item[(C3)] $|\nabla u^{*}|\leq \frac{a}{2}$ sur $\partial\Omega^{*}\cap\partial B$,
  \item[(C4)] $|\nabla u^{*}|\geq \frac{a}{2}$ sur $\partial\Omega^{*}\cap\partial\Omega$, et
  \item[(C5)] $|\nabla u^{*}|=\frac{a}{2}$ sur $\partial\Omega\setminus(\partial\Omega\cup\partial B)$.
\end{itemize}

Supposons par l'absurde que $\partial\Omega\cap\partial B\neq\emptyset$. Comme $B\subset\Omega^{*}\subset\Omega$, une des situations suivantes peut se produire :
\begin{enumerate}
  \item $\partial\Omega=\partial\Omega^{*}=\partial B$
  \item $\partial\Omega\neq\partial B$ et $\partial\Omega^{*}=\partial B$
  \item $\partial\Omega\neq\partial B$ et $\partial\Omega^{*}\neq\partial B$
  \item $\partial\Omega\neq\partial B$ et $\partial\Omega^{*}=\partial\Omega$
  \item $\partial\Omega\neq\partial B$ et $\partial\Omega^{*}\neq\partial\Omega$
\end{enumerate}

En appliquant le principe du maximum à $u$ et $u^{*}$ solutions du problème de Dirichlet sur (respectivement) $\Omega$ et $\Omega^{*}$, les conditions d'optimalité (C1),...,(C5) nous donnent, dans tous les cas de figure,
$$\frac{a}{2}\leq|\nabla u^{*}(x)|\leq|\nabla u(x)|\leq k,\;\forall x\in\partial\Omega\cap\partial\Omega^{*}\cap\partial B.$$
Obtenant ainsi une contradiction.
\begin{remark}
Dans \cite{bk}, les auteurs montrent que le problème $QS(a\delta_{C},k)$ admet une solution dès que $a>3.92k$.
\end{remark}

\subsection{Preuve du Théorème 3.2}

La démonstration de ce théorème utilise le Théorème 3.1 et la proposition suivante.
\begin{proposition}
Soit $\Omega$ le minimum de la fonctionnelle $J$ obtenu au Théorème 3.1. Si $\partial\Omega$) ne contient pas d'arc de type II, alors $\Omega$ est symétrique par rapport à l'axe $Ox$ et $u$ vérifie $u(x,-y)=u(x,y)$ pour tout $(x,y)\in \Omega$.
\end{proposition}
\begin{proof}
Pour $t$ petit, soit $\Omega^{t}$ le symétrisé de $\Omega$ par la symétrisation de Steiner continue. Appelons $u^{t}$
la fonction symétrisée de $u$ et $u_{t}$ le potentiel associé à $\Omega^{t}$, c'est à dire la solution du problème de Dirichlet $P(a\delta_{C})$ sur $\Omega^{t}$.
En utilisant la formulation variationnelle de l'équation aux dérivées partielles, $u_{t}$ minimise (sur $H^{1}_{0}(\Omega{t})$)
la fonctionnelle $$\int_{\Omega^{t}}|\nabla v|^2-2a\int_{C}v.$$
Par conséquent, puisque $u^{t}\in H^{1}_{0}(\Omega^{t})$\cite{br},
$$\int_{\Omega^{t}}|\nabla u_{t}|^2-2a\int_{C}u_{t}\leq \int_{\Omega^{t}}|\nabla u^{t}|^2-2a\int_{C}u^{t}.$$

Par ailleurs, d'après les propriétés classiques de la symétrisation \cite{br}, on a $|\Omega^{t}|=|\Omega|$,
$$\int_{\Omega^{t}}|\nabla u^{t}|^2\leq\int_{\Omega}|\nabla u|^2, \text{et}$$
$$\int_{\Omega^{t}}u^{t}=\int_{\Omega}u.$$
Donc
$$\int_{\Omega^{t}}|\nabla u^{t}|^2-2a\int_{C}u^{t}\leq \int_{\Omega^{t}}|\nabla u|^2-2a\int_{C}u.$$
Par suite,
$$J(\Omega^{t})\leq J(\Omega).$$
Comme $\Omega$ est un minimum de $J$ sur $\mathcal{O}_{C}$ qui contient aussi $\Omega^{t}$ pour $t$ petit, alors $J(\Omega^{t})=J(\Omega)$, et
$$\int_{\Omega^{t}}|\nabla u^{t}|^2=\int_{\Omega}|\nabla u|^2.$$
Ainsi $u^{t}=u_{t}$. Pour conclure, nous utilisons le résultat suivant, prouvé par Brock dans son Habilitation \cite{br} : Si
$$\lim_{t\rightarrow 0}\frac{\int_{\Omega^{t}}|\nabla u^{t}|^2-\int_{\Omega}|\nabla u|^2}{t}=0$$
alors $\Omega$ est localement symétrique dans la direction $Oy$ au sens de la définition suivante :
\begin{definition}
Soient $\Omega$ un ouvert de $\mathbb{R}^2$ et $u\in H^{1}_{0}(\Omega)\cap C^{1}_{0}(\bar\Omega)$.
On dit que $\Omega$ est localement symétrique par rapport à l'axe $\{y=y_{0}\}$ ($y_{0}\in\mathbb{R}$) si on peut le décomposer sous la forme
$$\Omega=\bigcup_{k=1}^{k=m}(\omega_{1}^{k}\cup\omega_{2}^{k})\cup G$$
($m$ éventuellement infini), où
\begin{enumerate}
  \item $\omega_{1}^{k}$ est une composante connexe (maximale)de $\Omega\cap\{\frac{\partial u}{\partial y}>0\}$,
  \item $\omega_{2}^{k}$ est son symétrique par rapport à l'axe $\{y=y_{0}\}$
  \item pour tout $(x',y)\in\omega_{1}^{k}$, $u(x',y)=u(x',2y_{0}-y)<u(x',z)$ pour tout $z\in[y,2y_{0}-y]$
  \item $\frac{\partial u}{\partial y}=0$ sur $G$.
\end{enumerate}
\end{definition}
 Soit $\omega_{1}$ une composante connexe de $\Omega$ définie comme ci-dessus et $\omega_{2}$ son image
 par une symétrie $\sigma$ d'axe $y=y_{0}$. Soit $x\in\Omega$. Posons $v(x)=u(\sigma(x))$ et $\tilde{u}(x)=u(x)-v(x)$, alors
 \begin{enumerate}
   \item $\Delta\tilde{u}(x)=a(\delta_{\sigma(C)}-\delta_{C})$ dans $\Omega$.
   \item $\tilde{u}$ est harmonique sur $\Omega\setminus (C\cup\sigma(C))$.
   \item $\tilde{u}=0$ sur l'ouvert $\omega_{1}$ par définition de la locale symétrie.
 \end{enumerate}
Par conséquent, d'après le principe d'unique continuation, $\tilde{u}$ est identiquement nulle sur $\Omega\setminus (C\cup\sigma(C))$,
ou encore $u=v$ sur $\Omega\setminus (C\cup\sigma(C))$. Donc, $\Omega\setminus (C\cup\sigma(C))$ est symétrique par rapport à l'axe $\{y=y_{0}\}$.
Supposons, par l'absurde, que $\sigma(C)\neq C$. Nous allons montrer que $u=v$ sur $\sigma(C)$.\\
Notons $w$ une solution fondamentale de l'équation $-\Delta w=a\delta_{C}$ dans $\mathbb{R}^2$ :
\begin{eqnarray*}
  w(x,y) & = &  \frac{a}{4\pi}[(x+1)\ln((x+1)^2+y^2)-(x-1)\ln((x-1)^2+y^2)] \\\\
   & - & y\frac{a}{\pi}[\arctan(\frac{x-1}{y})-\arctan(\frac{x+1}{y})-4].
\end{eqnarray*}
qui est continue par rapport à $y$. Comme $u-w$ est harmonique dans $\Omega$ alors $u$, comme $w$, est continue au passage de $C$.
Donc $u$ est continue sur $\Omega$. Soit maintenant $y\in\sigma(C)$ et soit $y_{n}$ une suite de points de $\Omega\setminus\sigma(C)$ qui tend vers $y$. La contiuité de $u$ et de $v$ entraînent qu'à la limite $u(y)=v(y)$, d'où $u=v$ sur $\sigma(C)$ (et de la même manière $u=v$ sur $C$). Par suite, $u=v$ sur tout $\Omega$. Mais $u$ est harmonique sur $\Omega\setminus\sigma(C)$ donc, en particulier, sur le demi-plan $y<y_{0}$. Donc $v$ est aussi harmonique sur $y<y_{0}$, ce qui est absurde si $y_{0}\neq 0$. On conclut donc que $y_{0}=0$, $\sigma(C)=C$ et $\Omega$ est symétrique par rapport à l'axe $\{y=0\} $.
\end{proof}

\section{Remarques finales}
\begin{remark}
On donne ici les grandes lignes de la démonstration du résultat de Shahgholian \cite{sh}. Soit $\Omega$ un ouvert (de classe $C^{2}$) solution du problème $QS(f,k)$ et soit $u\in C^{2}(\bar{\Omega})$ la solution du problème surdéterminé sur $\Omega$. Supposons qu'il existe un point $x\in\partial\Omega\setminus C$ dont la normale intérieure de vecteur directeur $\nu(x)$ ne rencontre pas le convexe $C$. On peut trouver alors, un hyperplan $T$ contenant $\nu(x)$ tel que $T\cap C=\emptyset$. On appelle $\eta$ le vecteur unitaire orthogonal à $T$ et dirigé dans le sens opposé à $C$. Dans ce cas, pour un certain $t_{0}$,
$$T=\{y\in\mathbb{R}^{N},\; \eta.y=t_{0}\}.$$
Posons alors $T_{t}\{y\in\mathbb{R}^{N};\; \eta.y=t\}$ de sorte que $T$ coincide avec $T_{t_{0}}$. Soit
$$\Omega_{t_{0}}=\{y\in\Omega,\; \eta.y>t_{0}\},$$
et soit $Ref(\Omega_{t_{0}})$ le symétrique de $\Omega_{t_{0}}$ par rapport à $T_{t_{0}}$. Celui-ci étant orthogonal à $\partial\Omega$ au point $x$ puisqu'il contient la normale intérieure à $\partial\Omega$ en $x$. Sans perdre de généralité, on prut toujours supposer que $t_{0}$ est le plus grand des réels $t$ pour lequel cette situation apparaît (puisque sinon, on se décalerait dans le sens du vecteur $\eta$ jusqu'à rencontrer un autre hyperplan $T_{t}$ orthogonal à $\partial\Omega$). Commençons par montrer le
\begin{lemma}
Le symétrique $Ref(\Omega_{t_{0}})$ est contenu dans $\Omega$.
\end{lemma}
\begin{proof}
Supposons, par l'absurde, que $Ref(\Omega_{t_{0}})$ n'est pas inclus dans $\Omega$. Soit
$$s=\sup\{t,\; Ref(\Omega_{t})\cap\Omega^{c}\neq\emptyset\}.$$
Alors $Ref(\Omega_{s})$ est dans $\Omega$ et son bord est intérieurement tangent à $\partial\Omega$ en un point $x^{0}$ n'appartenant pas à $T_{s}$. Définissons $\tilde{u}$ telle que $\tilde{u}(y)=u(y')$ ($y'$ étant le symétrique de $y$ par rapport à $T_{s}$). Si $v=u-\tilde{u}$, alors $v$ est surharmonique dans $Ref(\Omega_{s})$ et elle est positive sur $\partial Ref(\Omega_{s})$. Par conséquent, ou bien $v=0$ (ce qui n'est pas possible), ou bien $v>0$ sur $Ref(\Omega_{s})$ et atteint son maximum en un point du bord o\`{u} elle est nulle, en particulier en $x^{0}$. Mais alors, par le principe du maximum de Hopf
$$0>\frac{\partial v}{\partial\nu}(x^{0})=\frac{\partial u}{\partial\nu}(x^{0})-\frac{\partial\tilde{u}}{\partial\nu}(x^{0})=0.$$
Aboutissant ainsi à une contradiction.
\end{proof}
On montre alors, comme le fait Serrin \cite{se}, que $x$ est un zéro double de la fonction $v=u-\tilde{u}$ (cette fois-ci $\tilde{u}(y)=u(y')$ et $y'$ est le symétrique de $y$ par rapport à $T_{t_{0}}$). On termine la démonstration en utilisant le lemme classique de Serrin qui généralise le principe du maximum de Hopf à une situation comme celle-ci : si $v$ est de classe $C^{2}$ dans $\bar\Omega_{t_{0}}$ et
\begin{eqnarray*}
\Delta v& \leq & 0 \quad \text{dans  }\Omega_{t_{0}},\\
v& \geq & 0\text{ dans  }\Omega_{t_{0}}, et\\
 v(x)& = &0.
\end{eqnarray*}

Alors, pour toute direction $\delta$ entrant non tangentiellement dans $\Omega_{t_{0}}$, on a
\begin{itemize}
  \item[$\bullet$] ou bien $\frac{\partial v}{\partial\delta}(x)>0$
  \item[$\bullet$] ou bien $\frac{\partial^{2} v}{\partial\delta^{2}}(x)>0$ (sauf si $v\equiv 0$).
\end{itemize}

\end{remark}

\begin{remark}
On peut utiliser les résultats obtenus par la symétrisation de Steiner continue, pour retrouver la condition surdéterminée pour les arcs de Type I.
En effet, puisqu'on a deux arcs de cercle qui se font face, tous les deux de Type I et que $u$ est symétrique, alors la condition d'optimalité s'écrit :
$$\int_{\theta}^{\theta+\alpha}(k^2-|\nabla u(s)|^2)\phi(s)ds\geq 0,$$
pour toute fonction $\phi$ décroissante, et
$$\int_{-\theta}^{-\theta-\alpha}(k^2-|\nabla u(s)|^2)\phi(s)ds\geq 0$$
pour toute fonction $\phi$ croissante. Ce qui est équivalent, en utilisant la symétrie de $u$, à
$$\int_{-\theta}^{-\theta-\alpha}(k^2-|\nabla u(s)|^2)\phi(s)ds=0,$$
pour toute fonction $\phi$ croissante.
Soit maintenant $\phi$ une fonction de classe $C^1$, elle s'écrit trivialement $\phi=\Phi_{1}-\phi_{2}$ où $\phi_{1}$ et $\phi_{2}$
sont toutes les deux croissantes. Donc l'égalité ci-dessus est vraie pour toute fonction $\phi$ de classe $C^1$ et par conséquent $|\nabla u|=k$
presque partout sur tout arc de Type I.
\end{remark}
\begin{remark}
Dans leur article \cite{bsl}, en utilisant le principe du maximum combiné avec la condition de compatibilité du problème de Neumann, les auteurs donnent une condition suffisante d'existence de solution pour le problème $QS(f,k)$ : If $|\nabla u_{C}|>k$ alors $C\subset \Omega$, ($C$ étant l'enveloppe convexe de $K$). Dans le cas où la fonction $f$ est à symétrie radiale avec $C=B_{R}$, les auteurs obtiennent la condition nécessaire et suffisante suivante : $$\int_{0}^{R}s^{N-1}f(s)ds>kR^{N-1}.$$
Par ailleurs, en utilisant essentiellement le principe du maximum \cite{ba2}, l'auteur montre que le problème $QS(f,k)$ admet une solution si et seulement si $$\int_{C}f(x)dx>k|\partial C|.$$\\
\end{remark}
\begin{remark}
Dans leur article \cite{gs}, les auteurs montrent l'existence d'un minimum pour la fonctionnelle
$$J(v)=\int_{\mathbb{R}^N}(|\nabla v|^2-2fv+k^2\chi_{v>0})dx,$$
pour tout $0\leq v\in H^1(\mathbb{R}^N)$. Ils démontrent que $(\Omega_{u},u)$ est une solution du problème $QS(f,k)$ mais seulement dans un sens faible. La condition surdéterminée est donnée par :
$$\lim_{\epsilon\rightarrow 0}\int_{u>\epsilon}(|\nabla u|^2-k^2).\eta d\mathcal{H}^{N-1}=0,$$
pour tout $\eta\in C^{\infty}_{0}(\mathbb{R}^N ; \mathbb{R}^N)$.
Ils concluent leur article par donner la condition suffisante suivante : pour un $R>0$ donné, si $K\subset B_{R}$ et $\int_{B_{R}}f(x)dx>(\frac{6^NN}{3R}|B_{R}|)k$ alors $K\subset B_{3R}\subset \Omega_{u}$.\quad
Lorsque $f=a\delta_{C}$, les auteurs donnent $a>24\pi k$ comme condition suffisante d'existence de solutions pour le problème considéré.
\end{remark}

\end{document}